\documentclass[11pt]{article}
\usepackage{times}
\usepackage{amsmath,amssymb,amsthm,url,hyperref}

\newtheorem{theorem}{Theorem}[section]

\newtheorem{lemma}[theorem]{Lemma}

\newcommand{\rr}{\mathbb{R}}

\newcommand{\nn}{\mathbb{N}}

\nonstopmode \numberwithin{equation}{section}

\allowdisplaybreaks 
\begin{document}

\title{Quasi-orthogonality and zeros of some $\boldsymbol{_2\phi_2}$ and $\boldsymbol{_3\phi_2}$ polynomials }

\author{
  P P Kar$^{a,1},$ K Jordaan$^b,$ P Gochhayat$^{a,2},$ M K Nangho$^c$\\
$^{a}$Department of  Mathematics, Sambalpur University,
Sambalpur, 768019, Odisha, India\\
$^{1}$Email: pinakipk@gmail.com,~
$^{2}$Email: pgochhayat@gmail.com\\
$^b$Department of  Decision Sciences, University of South Africa,
Pretoria, 0003,
South Africa \\
Email: jordakh@unisa.ac.za\\
$^c$Department of Mathematics and Computer Science, University of Dschang, 67, Dschang, Cameroon \\
Email: maurice.kenfack@univ-dschang.org }

\maketitle

\begin{abstract}
\noindent We state and prove the $q$-extension of a result due to
Johnston and Jordaan (cf. \cite{Johnston-2015}) and make use of
this result, the orthogonality of $q$-Laguerre,
little $q$-Jacobi, $q$-Meixner and Al-Salam-Carlitz I polynomials as well as contiguous relations satisfied by the polynomials, to establish the quasi-orthogonality
of certain $_2\phi_2$ and $_3\phi_2$ polynomials. The location and interlacing properties of the real zeros of these quasi-orthogonal polynomials are studied. Interlacing properties of the zeros of $q$-Laguerre quasi-orthogonal polynomials $L_n^{(\delta)}(z;q)$ when $-2<\delta<-1$ with those of $L_{n-1}^{(\delta+1)}(z;q)$ and
$L_n^{(\delta+1)}(z;q)$ are also considered.
\end{abstract}


\vspace{.2in} \noindent {\it {Keywords:}} Basic hypergeometric
series, contiguous relations, quasi-orthogonal polynomials,
$q$-Laguerre polynomials, little $q$-Jacobi polynomials, $q$-Meixner polynomials, Al-Salam-Carlitz I polynomials,
interlacing of zeros.

\section{Introduction}\label{Section-1}

A sequence $\{P_n\}_{n=0}^\infty$ of real polynomials of exact degree $n$, $n=0,1,\dots,$ is orthogonal with respect to a positive weight function $w(x)$ on an interval $[a,b]$ if
$$\int_a^b x^k P_n(x)w(x)\; dx\begin{cases}
=0 \quad \textrm{for} \quad k=0,1,\ldots, n-1\\\neq 0\quad \textrm{for} \quad k=n.\end{cases}$$
A well-known consequence of orthogonality is that the $n$ zeros of $P_n(x)$ are real, simple and lie in the open interval $(a,b)$.
The zeros of $P_n$ depart from the orthogonality interval in a specific way when the parameters are changed to values where the polynomials are no longer orthogonal and this phenomenon can be explained in terms of the concept of quasi-orthogonality.
A polynomial sequence $\{R_n\}_{n=0}^{\infty}$, $\deg{R_n}=n$, $n\ge r$ is quasi-orthogonal of order $r$ where $n,r\in \nn$ with respect to $w(x)>0$ on $[a,b]$ if $$\int_a^b x^k R_n(x)w(x)\; dx \begin{cases} =0 \quad \textrm{for} \quad k=0,1,\ldots ,n-r-1\\ \neq 0 \quad \textrm{for} \quad k=n-r.\end{cases}$$
Shohat \cite{Shohat} showed that, if $\{P_n\}_{n=0}^\infty$ is a family of orthogonal polynomials with respect to $w(x)>0$ on $[a,b]$, then $\{R_n\}_{n=0}^{\infty}$ is quasi-orthogonal of order $r$ on $[a,b]$ with respect to $w(x)$ if and only if
    there exist constants $c_{n,i}$, $i=0,\ldots,n$ and $c_{n,0}c_{n,r}\neq0$ such that \[R_n(x)=\begin{cases}c_{n,0}P_{n}(x)+c_{n,1}P_{n-1}(x)+\dots+c_{n,r}P_{n-r}(x), \quad n\in\{r,r+1,\dots\}\\
    c_{n,0}P_n(x)+c_{n,1}P_{n-1}
    (x)+\dots+c_{n,n}P_0(x), \quad n\in\{0,\dots,r-1\}.\end{cases}\]
If $R_n$ is quasi-orthogonal of order $r$ on $[a, b]$ with respect to a positive weight function,
    then at least $(n -r)$ distinct zeros of $R_n$ lie in the interval $(a, b)$ (cf. \cite{BrezinskiDR,Shohat}).

\medskip \noindent The notion of quasi-orthogonal polynomials of order $1$ was introduced and studied by Riesz \cite{Riesz}. Quasi-orthogonality of order $2$ was introduced by Fej\'{e}r \cite{Fejer} but it was Shohat \cite{Shohat} and then Chihara
\cite{Chihara}, who generalized the concept of
quasi-orthogonality for any order. In \cite{Chihara}, Chihara investigated extremal properties and location of zeros of quasi-orthogonal polynomials and showed that such polynomials satisfy a three-term recurrence relation with polynomial coefficients. Other classical references on
quasi-orthogonality include the works by Dickinson
\cite{Dickinson}, Brezinski \cite{Brezinski}, Draux
\cite{Draux1,Draux3} and Ronveaux \cite{Ronveaux}. Properties of quasi-orthogonal polynomials and their zeros with respect to classical weights were studied in \cite{BrezinskiDR,BrezinskiDR-2019,Driver-2017, Driver-2018, DJ, DM1, DM3, Johnston-2016,JT,Joulak}.
Recent contributions on the topic may be found in the work by Tcheutia et al. \cite{Alta-SIGMA18} on quasi-orthogonality of classical polynomials on $q$-linear and $q$-quadratic
lattices and by Ismail and Wang \cite{IW} on a general theory for quasi-orthogonal polynomials, specifically their differential equations, discriminants and electrostatics.

\medskip \noindent Applications of zeros of quasi-orthogonal polynomials to interpolation theory, quadrature and approximation theory can be found in \cite{Brezinski-2010, bu, Zhang, Y1,Y2,Y3}, and references therein, while applications of algebraic and spectral properties of quasi-orthogonal polynomials to problems in quantum radiation are described in \cite{Zarzo}.

\medskip \noindent In this paper, we consider the quasi-orthogonality of some families of basic hypergeometric polynomials that do not appear in the $q$-Askey scheme of hypergeometric orthogonal polynomials. A basic hypergeometric function, a class of functions introduced and studied by Heine \cite{Heine}, is a power series in one complex variable $z$ with coefficients which depend, apart from $q$, on $r$ complex numerator
 parameters $\alpha_1,\alpha_2,\cdots,\alpha_r$ and $s$ complex
 denominator parameters $\beta_1,\beta_2,\cdots,\beta_s$. The basic hypergeometric series is defined as follows:
\begin{equation*}
_r\phi_s\left(\begin{matrix}\alpha_1,~\alpha_2,\cdots,
\alpha_r\\\beta_1,\beta_2,\cdots,\beta_s\end{matrix};q,z\right)=
\sum_{k=0}^{\infty}\frac{(\alpha_1;q)_k(\alpha_2;q)_k\cdots(\alpha_r;q)_k}{(q;q)_k(\beta_1;q)_k(\beta_2;q)_k\cdots(\beta_s;q)_k}\left\{(-1)^k
q^{\binom{k}{2}}\right\}^{1+s-r}z^k,
\end{equation*}
with $\binom{k}{2} = k(k-1)/2$ where $q\neq 0$ when $r>s+1$ and
$(a;q)_k$ is the $q$-shifted factorial defined by
\begin{align*}
(a;q)_k & =
\begin{cases}1;~&k=0,\\ (1-a)(1-aq)(1-aq^2)\cdots(1-aq^{k-1}) ; & k\in \mathbb{N}:=\{1,2,3,\cdots\}\end{cases}\\
\text{and}\ (a;q)_{\infty}&=\prod_{k=0}^{\infty}(1-aq^k).
\end{align*}
When one of the numerator parameters, say $\alpha_1$ is equal to
$q^{-n},$ then the series terminates and the function is a
polynomial of degree $n.$  We will assume throughout this paper
that $0<q<1$.

\medskip \noindent$q$-Laguerre polynomials, defined by (cf. \cite[pp. 522, Eqn. (14.21.1)]{Koekoek}) $$L_n^{(\delta)}(z;q)=\frac{\left(q^{\delta+1};\,q\right)_n}{\left(q;\,q\right)_n}\ _1\phi_1\left(\begin{matrix}q^{-n}\\q^{\delta+1}\end{matrix};q,-q^{n+\delta+1}z\right),$$  are orthogonal with respect to the weight function $\displaystyle{\frac{z^{\delta}}{(-z;q)_\infty}}$ on the interval  $(0,\infty)$ for $\delta>-1$ and hence we have
\begin{equation*}
\int_0^\infty
z^{j}{L}_{n}^{(\delta)}(z,q)\frac{z^{\delta}}{(-z;q)_\infty}dz=0;\quad
j=0,\cdots,n-1.
\end{equation*} $q$-Laguerre polynomials satisfy the relation (cf. \cite[Eqn. (4.12)]{Moak-1981}),
\begin{equation}\label{s1e2}
{L}_{n}^{(\delta-1)}(z,q)=q^{-n}{L}_{n}^{(\delta)}(z,q)-q^{-n}{L}_{n-1}^{(\delta)}(z,q).
\end{equation} Hahn \cite{Hahn} introduced the little $q$-Jacobi polynomials (cf. \cite[pp. 482, Eqn. (14.12.1)]{Koekoek}) given by $$p_n(z;\,a,\,b|q)=\ _2\phi_1\left(\begin{matrix}q^{-n},abq^{n+1}\\aq\end{matrix}q,\ qz\right).$$ Little $q$-Jacobi polynomials are discrete orthogonal with respect to the weight function $\displaystyle{\frac{a^zq^z(bq;q)_z}{(q;q)_z}}$ for \newline $z\in(0,1)$ , $0<aq <1$ and $bq < 1$.

\medskip \noindent $q$-Meixner polynomials (cf. \cite[pp. 488, Eqn. (14.13.1)]{Koekoek}) given by $$M_n(q^{-z};\,b,\,c;q)=\ _2\phi_1\left(\begin{matrix}q^{-n},q^{-z}\\bq\end{matrix};q,\ -{q^{n+1}\over c}\right),$$ are discrete orthogonal with respect to the weight function $\displaystyle{{(bq;q)_z c^z q^{\binom{z}{2}} \over (-bcq;q)_z (q;q)_z}}$ on the interval  $(0,\infty)$ for $0<bq <1$ and $c>0.$

\medskip \noindent
Al-Salam-Carlitz I polynomials (cf. \cite[pp. 534, Eqn. (14.24.1)]{Koekoek}) given by $$U_n^a (z;q)=\ (-a)^n q^{\binom{n}{2}}\  _2\phi_1\left(\begin{matrix}q^{-n},z^{-1}\\0\end{matrix};q,\ {qz\over a}\right),$$ are orthogonal with respect to the weight $\displaystyle{\left(qz,{qz\over a};q\right)_\infty}$ on $(a,1)$ for $a<0.$

\medskip \noindent In this paper we will use the orthogonality of $q$-Laguerre, little $q$-Jacobi, $q$-Meixner and Al-Salam-Carlitz I polynomials respectively to investigate the quasi-orthogonality of the basic hypergeometric polynomials
\begin{equation}\label{1}\phi_n^{(k)}(z):=\
_2\phi_2\left(\begin{matrix}q^{-n},q^{\gamma+k}\\
q^{\delta+1},q^\gamma\end{matrix};q,-q^{n+\delta+1}z\right),\end{equation}
\begin{equation}\label{2}\Phi_n^{(k)}(z):=\
_3\phi_2\left(\begin{matrix}q^{-n},q^{\gamma+k},abq^{n+1}\\
aq,q^\gamma\end{matrix};q,qz\right),\end{equation}
\begin{equation}\label{3} \varphi_n^{(k)}(z):=\
    _3\phi_2\left(\begin{matrix}q^{-n},q^{\gamma+k},q^{-z}\\
    bq,q^\gamma\end{matrix};q,-{q^{n+1}\over c}\right)\end{equation}
    and
    \begin{equation}\label{5}
    \varPhi_n^{(k)}(z):=\ _3\phi_2\left(\begin{matrix}q^{-n},q^{\gamma+k},z^{-1}\\0,q^\gamma\end{matrix};q,{qz\over a}\right).
    \end{equation}

\medskip \noindent In  Section \ref{Section-2}, we prove the $q$-extension of a result due to Johnston and Jordaan \cite{Johnston-2015} and use contiguous relations satisfied by $q$-Laguerre polynomials to prove quasi-orthogonality and interlacing properties of the zeros of quasi-orthogonal $\phi_n^{(1)}(z)$
and $\phi_n^{(2)}(z)$ polynomials with those of monic $q$-Laguerre polynomials. In Section \ref{Section-3} we consider the quasi-orthogonality of $\Phi_n^{(k)}(z),$ $\varphi_n^{(k)}(z)$ and $\varPhi_n^{(k)}(z)$, proving that polynomials $\Phi_n^{(1)}(z)$ and $\Phi_n^{(2)}(z)$ are quasi-orthogonal of order $1$ and $2$ respectively and establish the interlacing of the zeros of $\Phi_n^{(1)}(z)$ and $\Phi_n^{(2)}(z)$ with those of monic little $q$-Jacobi polynomials. Quasi-orthogonality and interlacing of zeros of $q$-Laguerre polynomials are investigated in Section \ref{Section-4}.

\section{Quasi-orthogonality and zeros of $\boldsymbol{_2\phi_2}$ polynomials}\label{Section-2}
The following theorem which is analogous to \cite[Theorem
2.1]{Johnston-2015} shows that the polynomials $_r\phi_s$ of
degree $n$ can be expressed as a linear combination of polynomials
of lesser degree with shifted parameters.
\begin{theorem}\label{s2t1}
Let $n\in \mathbb N;~k=1,2,\cdots,n-1;~|q|<1$ and
$\alpha_1,~\alpha_2,\cdots,
\alpha_r,~\beta_1,\beta_2,\cdots,\beta_s\in\mathbb R$ with
$\alpha_1,~\alpha_2,\cdots,
\alpha_r,~\beta_1,\beta_2,\cdots,\beta_s\notin\{0,-1,-2,\cdots,
-n\}$ and $\alpha_2\notin\{0,1,\cdots,k-1\}$. Then
\begin{align}\label{s2t1e1}
    {}_r \phi_s\left(\begin{matrix}&q^{-n},& &q^{\alpha_2+1},& &\alpha_3,& &\cdots,& &\alpha_r\\&\beta_1,&
&\beta_2,& &\beta_3,& &\cdots,& &\beta_s\end{matrix};q,
z\right)&=A_k\ {}_r \phi_s\left(\begin{matrix}&q^{-n+k},&
&q^{\alpha_2-k+1},& &\alpha_3,& &\cdots,& &\alpha_r\\&\beta_1,&
&\beta_2,& &\beta_3,& &\cdots,& &\beta_s\end{matrix};q,
z\right)+\cdots\nonumber\\ &+A_1\ {}_r
\phi_s\left(\begin{matrix}&q^{-n+1},& &q^{\alpha_2-k+1},&
&\alpha_3,& &\cdots,& &\alpha_r\\&\beta_1,& &\beta_2,& &\beta_3,&
&\cdots,& &\beta_s\end{matrix};q, z\right)\nonumber\\ &+A_0\ {}_r
\phi_s\left(\begin{matrix}&q^{-n},& &q^{\alpha_2-k+1},&
&\alpha_3,& &\cdots,& &\alpha_r\\&\beta_1,& &\beta_2,& &\beta_3,&
&\cdots,& &\beta_s\end{matrix};q, z\right),
\end{align}
for non-zero constants $A_i,i=0,1,2,\cdots,k$ depending on $n,
\alpha_2$ and $q$.
\end{theorem}

\begin{proof}
The contiguous relation (cf. \cite[pp. 25, Eqn. (C28)]{CK1}) for
$_r\phi_s$ functions with $a=q^{-n}$ and $b=q^{\alpha_2}$ becomes
\begin{multline}
_r\phi_s\left(\begin{matrix}q^{-n},q^{\alpha_2},\alpha_3,\cdots,\alpha_r\\\beta_1,\beta_2,\cdots,\beta_s\end{matrix};q,z\right)=
\frac{q^{-n}(1-q^{\alpha_2})}{(q^{-n}-q^{\alpha_2})}\ _r\phi_s\left(\begin{matrix}q^{-n},q^{\alpha_2} q,\alpha_3,\cdots,\alpha_r\\\beta_1,\beta_2,\cdots,\beta_s\end{matrix};q,z\right)\\
+\frac{(1-q^{-n})q^{\alpha_2}}{(-q^{-n}+q^{\alpha_2})}\
_r\phi_s\left(\begin{matrix}q^{-n} q,
q^{\alpha_2},\alpha_3,\cdots,\alpha_r\\\beta_1,\beta_2,\cdots,\beta_s\end{matrix};q,z\right),\nonumber
\end{multline}
which, on re-ordering of terms, can be written as
\begin{multline}
_r\phi_s\left(\begin{matrix}q^{-n},q^{\alpha_2+1},
\alpha_3,\cdots,\alpha_r\\\beta_1,\beta_2,\cdots,\beta_s\end{matrix};q,z\right)=
a_n^{\alpha_2}\ _r\phi_s\left(\begin{matrix}q^{-n+1},q^{\alpha_2}, \alpha_3,\cdots,\alpha_r\\\beta_1,\beta_2,\cdots,\beta_s\end{matrix};q,z\right)~~~~~~~~~~~~~~~~~~~~~~~~~~~~~~~~~~~~~~~~~~\\
+b_n^{\alpha_2}\ _r\phi_s\left(\begin{matrix}q^{-n},q^{\alpha_2},
\alpha_3,\cdots,\alpha_r\\\beta_1,\beta_2,\cdots,\beta_s\end{matrix};q,z\right),\label{s2t1e2}
\end{multline}
where
$$a_n^{\alpha_2}=\frac{q^{n+\alpha_2}(1-q^{-n})}{1-q^{\alpha_2}}\
\text{and}\
b_n^{\alpha_2}=\frac{q^n(q^{-n}-q^{\alpha_2})}{1-q^{\alpha_2}}.$$
For ease of notation, let us write \eqref{s2t1e2} as
\begin{equation}\label{s2t1e3}
_r\phi_s(q^{-n},q^{\alpha_2+1};z)=a_n^{\alpha}\ _r\phi_s(q^{-n+1},q^{\alpha_2};z)+b_n^{\alpha}\ _r\phi_s(q^{-n},q^{\alpha_2};z).
\end{equation}
Apply \eqref{s2t1e3} to both the polynomials on the right hand
side of \eqref{s2t1e3} to obtain
\begin{multline}
_r\phi_s(q^{-n},q^{\alpha_2+1};z)=a_n^{\alpha_2}a_{n-1}^{\alpha_2-1}\ _r\phi_s(q^{-n+2},q^{\alpha_2-1};z)\\
+\left(a_n^{\alpha_2}b_{n-1}^{\alpha_2-1}+b_n^{\alpha_2}a_n^{\alpha_2-1}\right)\ _r\phi_s(q^{-n+1},q^{\alpha_2-1};z)
+b_n^{\alpha_2}b_n^{\alpha_2-1}\ _r\phi_s(q^{-n},q^{\alpha_2-1};z).\label{s2t1e4}
\end{multline}

\noindent Similarly, applying \eqref{s2t1e3} to the polynomials on the right
hand side of \eqref{s2t1e4} up to $(k-1)$ times, yields
\eqref{s2t1e1}. For each $A_i$, $i=1,\dots,k$, we have a
denominator $(1-q^{\alpha_2})(1-q^{\alpha_2-1})\cdots
(1-q^{\alpha_2-k+1})$ and hence we assume
$\alpha_2\notin\{0,1,\cdots,k-1\}.$
\end{proof}
\noindent To study the quasi-orthogonality and interlacing
properties of $\phi_n^{(k)}(z)$, $\Phi_n^{(k)}(z)$, $\varphi_n^{(k)}(z)$ and $\varPhi_n^{(k)}(z)$ defined by
\eqref{1}, \eqref{2}, \eqref{3} and \eqref{5} we will make use of the contiguous
relations proved in the following lemma.

\begin{lemma}\hspace{1cm}\\ \label{s2L1}
The following relations hold true:
\begin{enumerate}
  \item[(i)] The $q$-Laguerre polynomials  $L_n^{(\delta)}(z;q)$ for $n\in\mathbb{N}$ and $\delta>-1$ satisfy the relations
  \begin{equation}\label{s2L1e1}
    (1+z)L_n^{(\delta)}(zq;\,q)=\frac{q^{n+1}-1}{q^{n+\delta+1}}L_{n+1}^{(\delta)}(z;\,q)+\frac{1}{q^{n+\delta+1}}L_{n}^{(\delta)}(z;\,q)
  \end{equation}
  and
  \begin{equation}
  L_{n}^{(\delta)}(z;q)=\frac{1}{q^n}L_{n}^{(\delta)}(zq;q)-\frac{1-q^{n+\delta}}{q^n}L_{n-1}^{(\delta)}(zq;q).\label{A}
  \end{equation}

  \item[(ii)] The little $q$-Jacobi polynomials $p_n(z;\,a,\,b|q)$ for $n\in\mathbb{N},\ 0<aq<1$ and $bq<1,$ satisfy the relations
  \begin{equation}\label{s2L1e2}
     (1-zqb)p_n(z;a,\,bq|q)={bq^{n+1}(aq^{n+1}-1)\over abq^{2n+2}-1}p_{n+1}(z;a,b|q)+{bq^{n+1}-1\over abq^{2n+2}-1}p_{n}(z;a,b|q)
  \end{equation}
  and
   \begin{equation}\label{s2L1e2a}
     p_n(z;\,a,\,b|q)={1-abq^{n+1}\over 1-abq^{2n+1}} p_{n}(z;\,a,\,bq|q)+{abq^{n+1}(1-q^n)\over 1-abq^{2n+1}} p_{n-1}(z;\,a,\,bq|q).
  \end{equation}
  \item[(iii)] The $q$-Meixner polynomials $M_n\left(q^{-z};\,b,\,c;q\right)$ for $n\in\mathbb{N},\ 0<bq<1$ and $c>0,$ satisfy the relation
    \begin{equation}\label{s2L1e2-1}
    \left(bc+q^{-z}\right)M_n\left(q^{-z};\,b,\,{c\over q};q\right)=
    {c(bq^{n+1}-1) \over q^{n+1}} M_{n+1}\left(q^{-z};\,b,\,c;q\right)+{q^{n+1}+c \over q^{n+1}}M_n\left(q^{-z};\,b,\,c;q\right).
    \end{equation}
\end{enumerate}
\end{lemma}

\begin{proof}
We begin by proving that for a fixed $n$,
$(1+z)L_n^{(\delta)}(zq;\,q)$ is a linear combination of
$L_{n+1}^{(\delta)}(z;\,q)$ and $L_{n}^{(\delta)}(z;\,q)$. Let $n$
be a positive integer, $\delta>-1$ and
$w(z;q)=\displaystyle{\frac{z^q}{(-z;q)_{\infty}}}$. The polynomial
$(1+z)L_n^{(\delta)}(zq;\,q)$ of degree $n+1$, can be written as
\[(1+z)L_{n}^{(\delta)}(zq;\,q)=\sum_{k=0}^{n+1}a_jL_{j}^{(\delta)}(z;\,q), ~~a_{n+1}\neq0,\]
where $a_j;$ $j=0,1,\dots,n+1$, is given by
\begin{equation}\label{s2L1e3}
a_j\int_{0}^{\infty}w(z;\,q)\left(L_j^{(\delta)}(z;\,q)\right)^2dz
=\int_{0}^{\infty}w(z;\,q)(1+z)L_{n}^{(\delta)}(zq;\,q)L_j^{(\delta)}(z;\,q)dz,
\end{equation}
since $(L^{(\delta)}_{j}(z;q))_{j=0}^{\infty}$ is orthogonal with
respect to the weight function $w(z;q)$ on the interval
$(0,\infty).$ Using the change of variable $x=zq$ and upon
substitution of the weight function, \eqref{s2L1e3} becomes
\begin{equation*}
a_j\int_{0}^{\infty}w(z;\,q)\left(L_j^{(\delta)}(z;\,q)\right)^2dz
=
q^{-\delta-1}\int_{0}^{+\infty}w(x;\,q)L_n^{(\delta)}(x;\,q)L_j^{(\delta)}\left({x\over
q};\,q\right)dx.
\end{equation*}
 Therefore, $a_j=0$ for $j=0,1,\dots,n-1$ and it follows that
 \begin{equation}\label{s2L1e4}
   (1+z)L_n^{(\delta)}(zq;q)=a_{n+1}L_{n+1}^{(\delta)}(z;q)+a_nL_n^{(\delta)}(z;q),
 \end{equation}
 where $a_{n+1}$ and $a_n$ are both different from zero.

\medskip \noindent Next, we compute the coefficients $a_{n+1}$ and $a_n$ by expanding both sides of \eqref{s2L1e4} in terms of powers of $z$
  and comparing the coefficients of $z^{n+1}$ and $z^n$ to have respectively
\begin{eqnarray*}
{(q^{-n};\,q)_n\left(q^{n+\delta+1}\right)^nq^{{n\choose 2}}\over
\left(q;\,q\right)_n^2} q^n
=a_{n+1}{(q^{-n-1};\,q)_{n+1}\left(q^{n+\delta+2}\right)^{n+1}q^{n(n+1)\over
2}\over \left(q;\,q\right)_{n+1}^2},
\end{eqnarray*}
and
\begin{eqnarray*}
  &&{\frac { \left( {q}^{\delta+1};q \right)_n  \left( {q}^{-n};q \right)_{n-1}  \left( {q}^{n+\delta+1}
 \right) ^{n-1}{q}^{{n-1\choose 2}}{q}^{n-1}}{
 \left( q;q\right)_n  \left( {q}^{\delta+1};q
 \right)_{n-1}  \left( q;q \right)_{n-1} }}+{\frac { \left( {q}^{-n};q\right)_n  \left( {q}^{n+\delta+1}
 \right) ^{n}{q}^{{n\choose 2}}{q}^{n}}{ \left(
 \left( q;q \right)_n  \right) ^{2}}}\\
&& ={\frac {a_{{n+1}} \left( {q}^{\delta+1};q \right)_{n+1}
 \left( {q}^{-n-1};q \right)_n  \left( {q}^{n+2+\delta} \right) ^{n}{q
}^{{n\choose 2}}}{ \left( q;q \right)_{n+1}  \left(
{q}^{\delta+1};q \right)_n
 \left( q;q \right)_n }}+{\frac {a_{{n}} \left( {q}^{
-n};q \right)_n  \left( {q}^{n+\delta+1} \right)
^{n}{q}^{{n\choose 2} }}{ \left(  \left( q;q \right)_n  \right)
^{2}}}.
   \end{eqnarray*}
Solving the above system of equations for $a_{n+1}$ and $a_n$, we
obtain $a_{n+1}=\displaystyle{q^{n+1}-1\over q^{n+\delta+1}}$ and
$a_{n}=\displaystyle{1\over q^{n+\delta+1}}.$ This completes the proof of the
first equation in Lemma \ref{s2L1}(i). The derivation of the other
equations in the lemma is similar, therefore it suffices to point
out that, once the relations have been obtained, they can easily
be verified by comparing the coefficient of $z^n$ on both sides.
Similarly, suitable applications of \cite[Eqn. (14.21.10)]{Koekoek} and \cite[Eqn. (4.12)]{Moak-1981},
\cite[Eqn. (14.21.8)]{Koekoek} and \cite[Eqn. (4.14)]{Moak-1981}, \cite[Eqn. (9)]{PKRS} and \cite[Eqn. (9b)]{Alta-2018},
and \cite[Eqn. (2.3)]{Kim} and \cite[Eqn. (10c)]{Alta-2018} respectively gives \eqref{s2L1e1}, \eqref{A}, \eqref{s2L1e2} and \eqref{s2L1e2-1}
while \eqref{s2L1e2a} is also proved in \cite[Eqn. (9b)]{Alta-2018}.
\end{proof}

\subsection{Quasi-orthogonality using $\boldsymbol{q}$-Laguerre polynomials}
\begin{theorem}\label{s2t2}
Let $n\in\mathbb N,\ k=1,2,\cdots,n-1,\ \delta,\gamma\in\mathbb
R,\ \delta>-1$ and $\gamma,\ \gamma+k\notin\{0,-1,-2,\cdots,-n\}.$
Then the polynomial $\phi_n^{(k)}(z)$
is quasi-orthogonal of order $k$ on $(0,\infty)$ with respect to the weight function $\displaystyle{\frac{z^{\delta}}{(-zq^k;q)_{\infty}}}$
and has at least $(n-k)$ distinct, real, positive zeros.
\end{theorem}

\begin{proof}
    Letting $r=2=s,\ \alpha_2=\gamma+k-1,\ \beta_1=q^{\delta+1},\ \beta_2=q^{\gamma}$ and replacing $z$ by $-q^{n+\delta+1}z$ in (\ref{s2t1e1}), there exist some constants $A_i,\ i\in\{0,1,\cdots,k\}$ such that
    \begin{equation}
    \phi_n^{(k)}(z)=\sum_{j=0}^{k}A_j\frac{(q;q)_{n-j}}{(q^{\delta+1};q)_{n-j}}L_{n-j}^{(\delta)}(zq^j;\,q).\label{s2t2e1}
    \end{equation}
    \noindent Iterating (\ref{s2L1e1}) yields
    \begin{equation}
    (-z;\,q)_jL_{n}^{(\delta)}(zq^j;\,q)= \sum_{i=0}^{j}a_{j,n+i}L_{n+i}^{(\delta)}(z;\,q).\label{s2t2e2}
    \end{equation}
    \noindent Multiplying \eqref{s2t2e1} by $(-z;\,q)_k$ and using \eqref{s2t2e2} as well as the relation $(-z;q)_k=(-z;\,q)_j(-q^jz;\,q)_{k-j},$ $0\leq j\leq k$ (cf. \cite[(I.20), pp. 352]{Gasper-book}), yields
    \begin{equation}\label{s2t2e3}
    (-z;q)_k\,\phi_n^{(k)}(z)=\sum_{i=0}^{k}g_{k-i}(z)L_{n-i}^{(\delta)}(z;q),
    \end{equation}
    where $g_{k-i}(z)=\sum_{l=i}^{k}A_l\displaystyle{(q;\,q)_{n-l}\over (q^{\delta+1};q)_{n-l}}a_{l,n-i}(-zq^l;\,q)_{k-l}$ is a polynomial of degree $k-i$, $i=0,\dots,k$ and $k=0,\dots,n-1$. Therefore, upon multiplication by the weight function $w(z;q)=\displaystyle{\frac{z^{\delta}}{(-z;q)_{\infty}}}$ and the factor $z^j$ on both sides of \eqref{s2t2e3}, and integrating with respect to z over the support $(0,\infty)$, we obtain
    $$
    \int_{0}^{\infty}z^j(-z;q)_k\,\phi_n^{(k)}(z)w(z;\,q)dz\begin{cases}=0;\qquad & j=0,\dots,n-k-1,\\ \neq 0; \quad & j=n-k.\end{cases}
   $$
\end{proof}

\subsubsection{Order 1}

Taking $k=1$ in Theorem \ref{s2t2}, the corresponding polynomial $\phi_n^{(1)}(z)$ is quasi-orthogonal of order $1$ on $(0,\infty)$
and hence $\phi_n^{(1)}(z)$  has at least $(n-1)$ real and distinct positive zeros.
In the next theorem we investigate the location of the $n$ real zeros $\phi_n^{(1)}(z)$  with respect to the endpoints of the orthogonality interval
and with respect to the zeros of monic $q$-Laguerre polynomials of degree $n$ and $n-1$.

\begin{theorem}\label{s2t3}
    Let $n\in\{2,3,\cdots\},$ $\delta>-1,$ $\gamma,\delta\in\mathbb{R}$ and $\gamma\notin\{0,-1,\cdots,-n-1\}.$
    Denote the zeros of $\phi_n^{(1)}(z)$ by $z_{i,n},\ i\in\{1,2,\cdots,n\}$ and those of the monic $q$-Laguerre polynomial
    $\displaystyle{\tilde{L}_n^{(\delta)}(zq;q)}$
    by $x_{i,n},\ i\in\{1,2,\cdots,n\}.$ Then
    \begin{itemize}
        \item[(a)] $z_{i,n}>0;$ $i\in\{1,\cdots,n\}$ if and only if $\gamma<-n$ or $\gamma>0.$
        \item[(b)] $z_{1,n}<0$ and $z_{i,n}>0;i\in\{2,\cdots,n\}$ if and only if $-n<\gamma<0.$
        \item[(c)] The zeros of $\phi_n^{(1)}(z)$ interlace with those of $\tilde{L}_{n}^{(\delta)}(zq;q)$ and $\tilde{L}_{n-1}^{(\delta)}(zq;q)$ as
        \begin{itemize}
            \item[(i)] $x_{i,n}<z_{i,n}<x_{i,n-1}$ for $i\in\{1,\cdots,n-1\}$ and $x_{n,n}<z_{n,n},$
            \item[(ii)] $z_{1,n}<x_{1,n}$ and $x_{i-1,n-1}<z_{i,n}<x_{i,n}$ for $i\in\{2,\cdots,n\}.$
        \end{itemize}
    \end{itemize}
\end{theorem}

\begin{proof}
    Letting $r=2=s,\ \alpha_2=\gamma,\ \beta_1=q^{\delta+1},\ \beta_2=q^{\gamma}$ and replacing $z$ by $-q^{n+\delta+1}z$ in \eqref{s2t1e2}, we obtain
    \begin{equation}\label{s2t3e2}
    \phi_n^{(1)}(z)=\frac{(1-q^{n+\gamma})}{(1-q^{\gamma})} \frac{(q;q)_n}{(q^{\delta+1};q)_n}L_n^{(\delta)}(z;q)
    +\frac{q^{\gamma}(q^{n}-1)}{(1-q^{\gamma})}\ \frac{(q;q)_{n-1}}{(q^{\delta+1};q)_{n-1}} L_{n-1}^{(\delta)}(zq;q).
    \end{equation}
    Using \eqref{A} in \eqref{s2t3e2} yields
    \begin{equation*}
    {q^{n}(1-q^\gamma)\over (1-q^{n+\gamma})}{(q^{\delta+1};q)_n\over (q;q)_n}\phi_n^{(1)}(z)
    =L_{n}^{(\delta)}(zq;q)+{(q^{n+\delta}-1) \over (1-q^{n+\gamma})} L_{n-1}^{(\delta)}(zq;q),
    \end{equation*}
    which in terms of monic $q$-Laguerre polynomials $\displaystyle{\tilde{L}_n^{(\delta)}(zq;q)=\frac{(-1)^n (q;q)_n}{q^{n (\delta +n+1)}}L_n^{(\delta)}(zq;q)}$ can be written as
    \begin{equation*}
    {(-1)^n (q^{\delta+1};q)_n\over q^{n(n+\delta)}}{(1-q^\gamma)\over (1-q^{n+\gamma})}\phi_n^{(1)}(z)
    =\tilde{L}_{n}^{(\delta)}(zq;q)+a_n \tilde{L}_{n-1}^{(\delta)}(zq;q),
    \end{equation*}
    where $a_n=\displaystyle{\frac{(1-q^{n+\delta})(1-q^n)}{q^{2n+\delta}(1-q^{n+\gamma})}}.$ Now set $f_n(z)=\displaystyle{\tilde{L}_{n}^{(\delta)}(zq;q) \over \tilde{L}_{n-1}^{(\delta)}(zq;q)}.$ Then $\displaystyle{\lim_{b\to \infty}f_n(b)=\infty}$ and $f_n(0)=\displaystyle{\frac{(q^{n+\delta}-1)}{q^{2n+\delta}}}.$
    \begin{itemize}
        \item[(a)] Since $f_n(0)<-a_n$ if and only if $\gamma<-n$ or $\gamma>0,$ so by \cite[Theorem 4(iii)]{Joulak}, $z_{i,n}>0$ for $i\in\{1,\cdots,n\}.$
        \item[(b)] Since $-a_n<f_n(0)<0$ if and only if $-n<\gamma<0,$ so by \cite[Theorem 4(i)]{Joulak}, $z_{1,n}<0$ and $z_{i,n}>0;i\in\{2,\cdots,n\}.$
        \item[(c)] Since $a_n<0$ for $\gamma<-n$ and $a_n>0$ for $\gamma>-n,$ so by \cite[Theorem 5(i),(ii)]{Joulak}, we have
        \begin{itemize}
            \item[(i)] $x_{i,n}<z_{i,n}<x_{i,n-1}$ for $i\in\{1,\cdots,n-1\}$ and $x_{n,n}<z_{n,n},$
            \item[(ii)] $z_{1,n}<x_{1,n}$ and $x_{i-1,n-1}<z_{i,n}<x_{i,n}$ for $i\in\{2,\cdots,n\}.$
        \end{itemize}
    \end{itemize}
\end{proof}

\subsubsection{Order 2}
\noindent Substituting $k=2$ in Theorem \ref{s2t2}, we know that
$\phi_n^{(2)}(z)$ is quasi-orthogonal polynomial of order $2$ and has
at least $(n-2)$ real, distinct, positive zeros. In addition, we
can also prove the following:
\begin{theorem}
    Let $n\in\{2,3,\cdots\},\ \gamma, \delta\in\mathbb{R}$ with $\delta>-1$ and $\gamma\notin\{0,-1,\cdots,-n-2\}.$
    Denote the zeros of $\phi_n^{(2)}(z)$ by $y_{i,n},\ i\in\{1,2,\cdots,n\}$ and
    those of monic $q$-Laguerre polynomials $\tilde{L}_{n-1}^{(\delta)}(zq^2;q)$ by $x_{i,n-1},\ i\in\{1,2,\cdots,n-1\}.$
    Then for $-n-1<\gamma<-n,$ the zeros of ${\phi}_n^{(2)}(z)$ are all real and distinct and at most two of them are negative.
    Furthermore, the zeros of ${\phi}_n^{(2)}(z)$ interlace with the zeros of
    $\displaystyle{\tilde{L}_{n-1}^{(\delta)}(zq^2;q)}$
    as follows:
    \[y_{1,n}<x_{1,n-1}\ and\ x_{i-1,n-1}<y_{i,n}<x_{i,n-1}\]
    for $i\in\{2,\cdots,n-1\}$ and $x_{n-1,n-1}<y_{n,n}.$
\end{theorem}

\begin{proof}
    Replacing $\alpha_2$ by $\alpha_2+1$ in \eqref{s2t1e4} and iterating, we obtain
    \begin{multline}
    \frac{q^{-2n}(1-q^{\alpha_2})(1-q^{\alpha_2+1})}{(q^{-n}-q^{\alpha_2})(q^{-n}-q^{\alpha_2+1})}\
    _r\phi_s\left(q^{-n},q^{\alpha_2+2};z\right)
    =\ _r\phi_s\left(q^{-n},q^{\alpha_2};z\right)
    +\frac{q^{\alpha_2}(1+q)(1-q^{-n})}{(q^{-n}-q^{\alpha_2+1})}\\
    \times\ _r\phi_s\left(q^{-n+1},q^{\alpha_2};z\right)
    +\frac{q^{2\alpha_2}(1-q^{-n})(1-q^{-n+1})}{(q^{-n}-q^{\alpha_2})(q^{-n}-q^{\alpha_2+1})}\
    _r\phi_s\left(q^{-n+2},q^{\alpha_2};z\right).\label{2.16}
    \end{multline}
    Substituting $r=2=s,\ \alpha_2=\gamma,\ \beta_1=q^{\delta+1},\ \beta_2=q^{\gamma}$ and replacing $z$ by $-q^{n+\delta+1}z$ in \eqref{2.16} , yields
    \begin{multline*}\nonumber
    {(1-q^\gamma)(1-q^{\gamma+1}) \over {(1-q^{n+\gamma})(1-q^{n+\gamma+1})}} \phi_n^{(2)}(z)
    ={(q;q)_{n}\over(q^{\delta+1};q)_{n}} L_n^{(\delta)}(z;q)
    +{q^{\gamma}(1+q)(q^{n}-1)\over(1-q^{n+\gamma+1})} {(q;q)_{n-1}\over(q^{\delta+1};q)_{n-1}}L_{n-1}^{(\delta)}(zq;q)\\
    +{q^{2\gamma+1}(1-q^{n})(1-q^{n-1})\over (1-q^{n+\gamma})(1-q^{n+\gamma+1})} {(q;q)_{n-2}\over(q^{\delta+1};q)_{n-2}}L_{n-2}^{(\delta)}(zq^2;q).
    \end{multline*}
    Using \eqref{A} to replace $L_{n}^{(\delta)}(z;q)$ and $L_{n-1}^{(\delta)}(zq;q)$ and writing the result in terms of monic $q$-Laguerre polynomials  $\displaystyle{\tilde{L}_{n}^{(\delta)}(zq^2;q)}=\frac{(-1)^{n} (q;q)_{n}}{q^{n(n+\delta+2)}}L_{n}^{(\delta)}(zq^2;q),$ we obtain
    \begin{multline*}
    {(-1)^n (q^{\delta+1};q)_n \over q^{n(n+\delta)}} {(1-q^\gamma)(1-q^{\gamma+1}) \over {(1-q^{n+\gamma})(1-q^{n+\gamma+1})}} \phi_n^{(2)}(z)
    =\tilde{L}_n^{(\delta)}(zq^2;q)\\
    +{(1-q^{n+\delta})(1+q)(1-q^{n})\over q^{2n+\delta+1}(1-q^{n+\gamma+1})} \tilde{L}_{n-1}^{(\delta)}(zq^2;q)+b_n \tilde{L}_{n-2}^{(\delta)}(zq^2;q),
    \end{multline*}
    where $$b_n={q(1-q^{n+\delta})(1-q^{n+\delta-1})(1-q^n)(1-q^{n-1})\over q^{4n+2\delta}(1-q^{n+\gamma})(1-q^{n+\gamma+1})}.$$
    Since
    $b_n<0$ if and only if $-n-1<\gamma<-n,$ by \cite[Theorem 4 and 5]{BrezinskiDR}, the results hold.
\end{proof}

\section{Quasi-orthogonality of $\boldsymbol{_3\phi_2}$ polynomials}\label{Section-3}

\subsection{Quasi-orthogonality using little $\boldsymbol{q}$-Jacobi polynomials}
We begin by considering the quasi-orthogonality of a class of $_3\phi_2$ polynomials.

\begin{theorem}
    Let $n\in\mathbb{N},\ k=1,\cdots,n-1,\ 0<q<1,\ 0<aq<1,\ bq<1,$ $\gamma\in\mathbb{R}$ and $\gamma,\ \gamma+k\notin\{0,-1,-2,\cdots,-n\}.$ Then the polynomial $\Phi_n^{(k)}(z)$
    is quasi-orthogonal of order $k$ with respect to the weight function $\displaystyle{(bq;q)_z (aq)^z \over (q;q)_z}(zbq;q)_k$ on the interval of orthogonality $(0,1)$ and has at least $(n-k)$ distinct real positive zeros.
\end{theorem}

\begin{proof} When $k=1$, letting $r=3,\ s=2,\ \alpha_2=\gamma,\ \alpha_3=abq^{n+1},\ \beta_1=aq,\ \beta_2=q^\gamma$ and replacing $z$ by $qz$ in \eqref{s2t1e2}, we have
    \begin{equation}\Phi_n^{(1)}(z)=\frac{q^n(q^{-n}-q^\gamma)}{(1-q^\gamma)}p_n(z;a,b|q)
    +\frac{q^{n+\gamma}(1-q^{-n})}{(1-q^\gamma)}p_{n-1}(z;a,bq|q),\label{s3t1e1}
    \end{equation}
    and substituting \eqref{s2L1e2} into \eqref{s3t1e1}, yields
    \begin{equation}
    \Phi_n^{(1)}(z)=q_1(z)p_n(z;a,b|q)
    +q_0(z)p_{n-1}(z;a,b|q),\label{s3t1e2}
    \end{equation}
    where \begin{align*}q_1(z)&=\frac{q^{n+\gamma}}{(1-q^\gamma)    (1-zqb)}\left[(q^{-n-\gamma}-1)(1-zqb)+\frac{b(q^n-1)(aq^n-1)}{(abq^{2n}-1)}\right]\\ \text{and}\ q_0(z)&=\frac{q^{\gamma}(q^{n}-1)(bq^n-1)}{   (1-zqb)(1-q^\gamma)(abq^{2n}-1)}.
    \end{align*}
   For $k=2$, substituting $r=3,\ s=2,\ \alpha_2=\gamma,\ \alpha_3=abq^{n+1},\ \beta_1=aq,\ \beta_2=q^\gamma$ and replacing $z$ by $qz$ in \eqref{2.16}, we obtain
    \begin{multline}
    {q^{-2n}(1-q^\gamma)(1-q^{\gamma+1})\over(q^{-n}-q^\gamma)(q^{-n}-q^{\gamma+1})}\ \Phi_n^{(2)}(z)
    =p_n(z;a,b|q)+{q^\gamma(1+q)(1-q^{-n})\over(q^{-n}-q^{\gamma+1})}\\
    \times\ p_{n-1}(z;a,bq|q) +{q^{2\gamma}(1-q^{-n})(1-q^{-n+1})\over(q^{-n}-q^\gamma)(q^{-n}-q^{\gamma+1})}p_{n-2}(z;a,bq^2|q).\label{s3t2e1}
    \end{multline}
    Substituting the values of $p_{n-1}(z;a,bq|q)$ and $p_{n-2}(z;a,bq^2|q)$ from \eqref{s2L1e2} into \eqref{s3t2e1}, yields
    \begin{multline}
    {(1-q^\gamma)(1-q^{\gamma+1})\over(1-q^{n+\gamma})(1-q^{n+\gamma+1})}\ \Phi_n^{(2)}(z)
    ={[A_1(z)+B_1(z)+C_1(z)]\over A_1(z)}p_n(z;a,b|q)\\
    +{q^\gamma(1-q^{n})(1-bq^n)\over(1-zbq)(1-q^{n+\gamma+1})}\bigg[{(1+q)\over(abq^{2n}-1)}
    +{bq^{n+\gamma+1}\over (1-zbq^2)(1-q^{n+\gamma})}\\
    \times\ {(1-q^{n-1})(1-aq^{n-1})\over(1-abq^{2n})(1-abq^{2n-1})}
    +{bq^{n+\gamma}(1-q^{n-1})(1-aq^{n-1})\over(1-zbq^2)(1-q^{n+\gamma})(1-abq^{2n-1})(1-abq^{2n-2})}\bigg]\\
    \times\ p_{n-1}(z;a,b|q)
    +{q^{2\gamma+1}(1-q^{n})(1-q^{n-1})\over(1-zbq)(1-zbq^2)(1-q^{n+\gamma})(1-q^{n+\gamma+1})}\\
    \times\ {(1-bq^n)(1-bq^{n-1})\over(1-abq^{2n-1})(1-abq^{2n-2})} p_{n-2}(z;a,b|q),\label{s3t2e2}
    \end{multline}
    where
    \begin{align*}
        &A_1(z)=(1-zbq)(1-zbq^2)(q^{-n}-q^\gamma)(q^{-n}-q^{\gamma+1})(abq^{2n}-1)(abq^{2n-1}-1),\\
        &B_1(z)=bq^{n+\gamma}(1-zbq^2)(1+q)(aq^n-1)(1-q^{-n})(q^{-n}-q^\gamma)(abq^{2n-1}-1),\\
        \text{and}&\\
        &C_1(z)=b^2q^{2n+2\gamma}(aq^n-1)(aq^{n-1}-1)(1-q^{-n})(1-q^{-n+1}).
        \end{align*}
For general $k=1,\dots,n-1$, we let $r=3,\ s=2$,
$\alpha_2=\gamma+k-1,\ \alpha_3=abq^{n+1},\ \beta_1=aq,\
\beta_2=q^\gamma$ and replace $z$ by $qz$ in \eqref{s2t1e1} to
obtain
\begin{multline}
\Phi_n^{(k)}(z)=A_0\ _2\phi_1\left(\begin{matrix}q^{-n},abq^{n+1}\\aq\end{matrix};q,qz\right)
+A_1\ _2\phi_1\left(\begin{matrix}q^{-n+1},abq^{n+1}\\aq\end{matrix};q,qz\right)+\cdots\\
+A_k\ _2\phi_1\left(\begin{matrix}q^{-n+k},abq^{n+1}\\aq\end{matrix};q,qz\right)
=\sum_{j=0}^k A_j\ p_{n-j}(z;a,bq^j|q),\label{a1}
\end{multline}
where $A_i,i=0,1,2,\cdots,k$ are non-zero constants depending on
$n, \gamma$ and $q$. Using \eqref{s2L1e2} to substitute the values
of $p_{n-j}(z;a,bq^j|q)$ iteratively into \eqref{a1} for each
$j=\{0,1,\cdots,k\}$, we obtain
    \begin{equation}\label{a2}
    \Phi_n^{(k)}(z)=\sum_{i=0}^{k}q_{n-i}(z)p_{n-i}(z;a,b|q),
    \end{equation}
    where each coefficient $q_{n-i}(z)$ in \eqref{a2} has a factor of the form $(zbq;q)_k$ in the denominator. Hence, multiplying by $(zbq;q)_k$ on both sides of \eqref{a2}, we see that
    \begin{equation}\label{a3}
    (zbq;q)_k\Phi_n^{(k)}(z)=\sum_{i=0}^{k}r_{n-i}(z)p_{n-i}(z;a,b|q),
    \end{equation}
    where the coefficient $r_{n-k}(z)$ is constant. Multiplying $w(z)z^t$ where $t\in\{0,1,\cdots,n-k\}$ and $w(z)=\displaystyle{(bq;q)_z (aq)^z \over (q;q)_z}$ is the little $q$-Jacobi weight, and taking summation on both sides of \eqref{a3}, we get
    \begin{eqnarray*}
        \sum_{z=0}^{\infty}z^t w(z) (zbq;q)_k\Phi_n^{(k)}(z) &=& \sum_{z=0}^{\infty} \sum_{i=0}^{k}z^t w(z) r_{n-i}(z)p_{n-i}(z;a,b|q)\\
        && \begin{cases}=0;\quad t=0,1,\cdots, n-k-1,\\ \neq 0 ;\quad t=n-k.\end{cases}
    \end{eqnarray*}
    This proves that $\Phi_n^{(k)}(z)$ is quasi-orthogonal of order $k$ and therefore it has $(n-k)$ real positive zeros in $(0,1)$ with respect to the weight function $\displaystyle{(bq;q)_z (aq)^z \over (q;q)_z}(zbq;q)_k.$
\end{proof}

\subsubsection{Order 1}

\begin{theorem}
    Let $n\in\mathbb{N},~0 < q < 1,~0 < aq <1,~bq <1$ and $\gamma\in\mathbb{R}$ with $\gamma\notin\{0,-1,\cdots,-n-1\}.$ Denote the zeros of
    $\Phi_n^{(1)}(z)$ by $Z_{i,n},\ i\in\{1,2,\cdots,n\}$ and those of the monic little $q$-Jacobi polynomial
    $\displaystyle{\tilde{p}_n(z;a,bq|q)}$
 by $X_{i,n},\ i\in\{1,2,\cdots,n\}.$
Then
    \begin{itemize}
        \item[(a)] $Z_{1,n}<0$ and $Z_{i,n}\in(0,1)$ for $i\in\{2,\cdots,n\}$ if and only if $-n<\gamma<0$.
        \item[(b)] $Z_{i,n}\in(0,1)$ for $i\in\{1,\cdots,n\}$ when $0<\gamma<n-1$ or $\gamma>n-1$ and $b<0$.
        \item[(c)] The zeros of $\Phi_n^{(1)}(z)$ interlace with those of $\tilde{p}_{n}(z;a,b q|q)$ and $\tilde{p}_{n-1}(z;a,b q|q)$ as follows:
        \begin{itemize}
            \item[(i)] $Z_{1,n}<X_{1,n}$ and $X_{i-1,n-1}<Z_{i,n}<X_{i,n}$ for $i\in\{2,\cdots,n\}$ when $-n<\gamma<n-1$ or $\gamma>n-1$ and $b<0$; \item[(ii)] $X_{i,n}<Z_{i,n}<X_{i,n-1}$ for $i\in\{1,\cdots,n-1\}$ and $X_{n,n}<Z_{n,n}$
            when $\gamma<-n$ or $\displaystyle{\gamma>n-1+\frac{\log abq^2}{\log q}}$ and $b>0$.
        \end{itemize}
    \end{itemize}
\end{theorem}

\begin{proof}
Using \eqref{s2L1e2a}, \eqref{s3t1e1} can be written in terms of monic little $q$-Jacobi polynomials \newline $\displaystyle{\tilde{p}_n(z;a,bq|q)
= \frac{(-1)^n q^{n(n-1)\over 2}(aq;q)_n}{ (abq^{n+2};q)_n}p_n(z;a,b q|q)}$ as
    \begin{equation*}
    (-1)^n q^{n^2-n\over 2} {(aq;q)_n \over (abq^{n+2};q)_n} {(1-q^\gamma)(1-abq^{2n+1})\over(1-q^{n+\gamma})(1-abq^{n+1})} \Phi_n^{(1)}(z)= \tilde{p}_{n}(z;a,bq|q)
    + A_n \tilde{p}_{n-1}(z;a,bq|q),
    \end{equation*}
    where \begin{equation} A_n={q^{n+\gamma-1}(1-q^n)(1-aq^n)(1-abq^{n-\gamma+1}) \over (1-q^{n+\gamma})(1-abq^{2n})(1-abq^{2n+1})}.\label{An}\end{equation} Now set $ \displaystyle{F_n(z)={\tilde{p}_{n}(z;a,bq|q) \over \tilde{p}_{n-1}(z;a,bq|q)}},$ so that \begin{equation}\label{fn0}F_n(0)=-{q^{n-1} (1-aq^n)(1-abq^{n+1}) \over (1-abq^{2n})(1-abq^{2n+1})}\end{equation} and \begin{equation}
    \label{fn1}
    F_n(1)={aq^{2n-1} (1-bq^{n+1})(1-abq^{n+1}) \over (1-abq^{2n})(1-abq^{2n+1})}.\end{equation} It follows immediately from the assumptions $n\in\mathbb{N},~0 < q < 1,~0 < aq <1,~bq <1$ that $F_n(0)<0$ and $F_n(1)>0$.
    \begin{itemize}
        \item[(a)]  Using \eqref{An} and \eqref{fn0}, we see that $-A_n<F_n(0)$ if and only if \begin{align*}
        \frac{q^{\gamma}(1-q^n)(1-abq^{n-\gamma+1})}{(1-q^{n+\gamma})}>(1-abq^{n+1}).
        \end{align*} Therefore, $-A_n<F_n(0)$ if and only if $\gamma>-n$ and \begin{equation}(1-q^n)(q^{\gamma}-abq^{n+1})>(1-abq^{n+1})(1-q^{n+\gamma})\label{e2}.\end{equation} Since \eqref{e2} is equivalent to $$q^{\gamma}(1-abq^{2n+1})>(1-abq^{2n+1}),$$ it follows that $-A_n<F_n(0)$ if and only if $-n<\gamma<0$. The result now follows from \cite[Theorem 4(i)]{Joulak}.
        \item[(b)] It is easy to see that $F_n(0)<-A_n$ if and only if $\gamma<-n$ or $\gamma>0$.
        To establish the relationship between $A_n$ and $F_n(1)$, consider
        \begin{equation}\frac{A_n}{F_n(1)}=\frac{\left(1-q^n\right) \left(1-a q^n\right) \left(1-a b q^{n-\gamma +1}\right)}{a q^{n-\gamma } \left(1-b q^{n+1}\right) \left(1-q^{\gamma +n}\right) \left(1-a b q^{n+1}\right)}\end{equation} where  $(1-q^{\gamma+n})>0$ if and only if $\gamma>-n$ while $(1-abq^{n-\gamma+1})>0$ when $-n<\gamma<n-1$  or $\gamma>n-1$ and $b<0$. Since there is no solution for $(1-q^{\gamma+n})<0$ and $(1-abq^{n-\gamma+1})<0$ we have $\displaystyle{\frac{-A_n}{F_n(1)}<0<1}$ when $-n<\gamma<n-1$ or $\gamma>n-1$ and $b<0$ and therefore $F_n(0)<-A_n<F_n(1)$ when $0<\gamma<n-1$ or $\gamma>n-1$ and $b<0$. The result follows from \cite[Theorem 4(iii)]{Joulak}.
        \item[(c)] Since $A_n>0$ if and only if $\displaystyle{\frac{1-abq^{n-\gamma+1}}{1-q^{n+\gamma}}>0}$, it follows, as above, that $A_n>0$ when $-n<\gamma<n-1$ or $\gamma>n-1$ and $b<0$. $A_n<0$ when  $\gamma<-n$ or $\displaystyle{\gamma>n-1+\frac{\log abq^2}{\log q}}$ and $b>0$.
The result follows from \cite[Theorem 5(i),(ii)]{Joulak}.
\end{itemize}
\end{proof}

\subsubsection{Order 2}

\begin{theorem}
    Let $n\in\{2,3,\cdots\},\ 0<aq<1,\ bq<1$ and $-n\neq\gamma\in\mathbb{R}$ with $\gamma\notin\{0,-1,\cdots,-n-2\}.$
Denote the zeros of $\Phi_n^{(2)}(z)$ by $Y_{i,n},\ i\in\{1,2,\cdots,n\}$ and
those of monic little $q$-Jacobi polynomials $\tilde{p}_{n-1}(z;a,bq^2|q)$ by $X_{i,n-1},\ i\in\{1,2,\cdots,n-1\}.$
Then for $-n-1<\gamma<-n,$ the zeros of ${\Phi}_n^{(2)}(z)$ are real and distinct and at most two of them are negative.
Furthermore, the zeros of ${\Phi}_n^{(2)}(z)$ interlace with the zeros of
    $\displaystyle{\tilde{p}_{n-1}(z;a,bq^2|q)}$
as follows:
    \[Y_{1,n}<X_{1,n-1},\ X_{i-1,n-1}<Y_{i,n}<X_{i,n-1}\ for\ i\in\{2,\cdots,n-1\}\ and\ X_{n-1,n-1}<Y_{n,n}.\]
\end{theorem}

\begin{proof}
Substituting the expressions for $p_{n}(z;a,b|q)$ and $p_{n-1}(z;a,bq|q)$ from \eqref{s2L1e2a} into \eqref{s3t2e1} and
   writing it in terms of monic little $q$-Jacobi polynomials \newline  $\displaystyle{\tilde{p}_{n}(z;a,bq^2|q)= \frac{(-1)^{n} q^{n(n-1)\over 2}(aq;q)_{n}}{(abq^{n+3};q)_{n}} p_{n}(z;a,bq^2|q)},$ we obtain
    \begin{multline*}
    (-1)^n q^{n^2-n\over 2} {(aq;q)_n \over (abq^{n+3};q)_n} {(1-abq^{2n+1})(1-abq^{2n+2})\over (1-abq^{n+1})(1-abq^{n+2})}
    {(1-q^\gamma)(1-q^{\gamma+1}) \over (1-q^{n+\gamma})(1-q^{n+\gamma+1})} \Phi_n^{(2)}(z)=\tilde{p}_{n}(z;a,bq^2|q)\\
    -\bigg[{abq^{2n+1}(1-q^n)(1-aq^n)\over (1-abq^{2n+1})(1-abq^{2n+2})}+{abq^{2n}(1-q^n)(1-aq^n)\over (1-abq^{2n})(1-abq^{2n+1})}\\
    -{q^{n+\gamma-1}(1+q)(1-q^n)(1-aq^n)\over (1-abq^{2n})(1-q^{n+\gamma+1})}\bigg] \tilde{p}_{n-1}(z;a,bq^2|q)\\
    +\left[{a^2 b^2 q^{4n-1} A_1\over (1-abq^{2n})(1-abq^{2n+1})}-{abq^{3n+\gamma-2}(1+q)A_1\over(1-abq^{2n})(1-q^{n+\gamma+1})}
    +{q^{2n+2\gamma-2}A_1\over (1-q^{n+\gamma})(1-q^{n+\gamma+1})}\right] \tilde{p}_{n-2}(z;a,bq^2|q),
    \end{multline*}
    where
    \[A_1 = {(1-q^{n})(1-q^{n-1})(1-aq^{n})(1-aq^{n-1})\over (1-abq^{2n-1})(1-abq^{2n})}.\]

    \noindent After simplification, we get
    \begin{multline*}
    (-1)^n q^{n^2-n\over 2} {(aq;q)_n \over (abq^{n+3};q)_n} {(1-abq^{2n+1})(1-abq^{2n+2})\over (1-abq^{n+1})(1-abq^{n+2})}
    {(1-q^\gamma)(1-q^{\gamma+1}) \over (1-q^{n+\gamma})(1-q^{n+\gamma+1})} \Phi_n^{(2)}(z)=\tilde{p}_{n}(z;a,bq^2|q)\\
    +{q^{n-1+\gamma} (1+q)(1-abq^{2n+1})(1-abq^{n+1-\gamma})(1-q^n)(1-aq^n)\over (1-abq^{2n})(1-abq^{2n+1})(1-abq^{2n+2})(1-q^{n+\gamma+1})} \tilde{p}_{n-1}(z;a,bq^2|q)\\
    +{q^{2n+2\gamma-2}(1-abq^{n-\gamma})(1-abq^{n+1-\gamma})A_1\over (1-abq^{2n})(1-abq^{2n+1})(1-q^{n+\gamma})(1-q^{n+\gamma+1})}
    \tilde{p}_{n-2}(z;a,bq^2|q).
    \end{multline*}
    Since the coefficient of $\tilde{p}_{n-2}(z;a,bq^2|q)$ is negative if and only if $-n-1<\gamma<-n,$ the result follows from \cite[Theorem 4 and 5]{BrezinskiDR}.
\end{proof}

\subsection{Quasi-orthogonality using $\boldsymbol{q}$-Meixner polynomials}

    \begin{theorem}
        Let $n\in\nn,\ k=1,\cdots,n-1,\ 0<bq<1,\ c>0,$ $\gamma\in\mathbb{R}$ and $\gamma,\ \gamma+k\notin\{0,-1,-2,\cdots,-n\}.$
Then the polynomial $\varphi_n^{(k)}(z)$
is quasi-orthogonal of order $k$ with respect to the weight function \newline $\displaystyle{{(bq;q)_z c^z q^{\binom{z}{2}} \over (-bcq;q)_z (q;q)_z}{(-bcq^{z-k+1};q)_k\over q^{kz}}}$ on the interval of orthogonality $(0,\infty)$ and has at least $(n-k)$ distinct real positive zeros.
    \end{theorem}

    \begin{proof}
        Letting $r=3,\ s=2$,
        $\alpha_2=\gamma+k-1,\ \alpha_3=q^{-z},\ \beta_1=bq,\
        \beta_2=q^\gamma$ and $z=-\displaystyle{q^{n+1}\over c}$ in \eqref{s2t1e1}, we get
        \begin{multline}
        \varphi_n^{(k)}(z)=A_0\ _2\phi_1\left(\begin{matrix}q^{-n},q^{-z}\\bq\end{matrix};q,-{q^{n+1}\over c}\right)
        +A_1\ _2\phi_1\left(\begin{matrix}q^{-n+1},q^{-z}\\bq\end{matrix};q,-{q^{n+1}\over c}\right)+\cdots\\
        +A_k\ _2\phi_1\left(\begin{matrix}q^{-n+k},q^{-z}\\bq\end{matrix};q,-{q^{n+1}\over c}\right)
        =\sum_{j=0}^k A_j\ M_{n-j}\left(q^{-z};b,{c\over q^j};q\right),\label{m1}
        \end{multline}
        where $A_i,i=0,1,2,\cdots,k$ are non-zero constants depending on $n, \gamma$ and $q$.
Substituting the values of $M_{n-j}\left(q^{-z};b,{c\over q^j};q\right)$ iteratively from \eqref{s2L1e2-1} in \eqref{m1}, we obtain
        \begin{equation}\label{m2}
        \varphi_n^{(k)}(z)=\sum_{i=0}^{k}u_{n-i}(z)M_{n-i}\left(q^{-z};b,c;q\right),
        \end{equation}
        where each coefficient $u_{n-i}(z)$ in \eqref{m2} has a factor of the form $\displaystyle{q^{kz}\over (-bcq^{z-k+1};q)_k}.$
Hence, dividing this factor on both sides of \eqref{m2}, yields
        \begin{equation}\label{m3}
        {(-bcq^{z-k+1};q)_k\over q^{kz}}\varphi_n^{(k)}(z)=\sum_{i=0}^{k}v_{n-i}(z)M_{n-i}\left(q^{-z};b,c;q\right),
        \end{equation}
        where the coefficient $v_{n-k}(z)$ is a constant. Multiplying $z^t$ where $t\in\{0,1,\cdots,n-k\}$ and the $q$-Meixner weight $w(z)=\displaystyle{(bq;q)_z c^z q^{\binom{z}{2}} \over (-bcq;q)_z (q;q)_z},$ and taking summation on both sides of \eqref{m3}, we get
        \begin{eqnarray*}
            \sum_{z=0}^{\infty}z^t w(z) {\left(-bcq^{z-k+1};q\right)_k\over q^{kz}}\varphi_n^{(k)}(z) &=&
            \sum_{z=0}^{\infty} \sum_{i=0}^{k} z^t w(z) v_{n-i}(z)M_{n-i}\left(q^{-z};b,c;q\right)\\
            && \begin{cases}=0;\quad t=0,1,\cdots, n-k-1,\\ \neq 0 ;\quad t=n-k.\end{cases}
        \end{eqnarray*}
This proves that $\varphi_n^{(k)}(z)$ is quasi-orthogonal of order $k$ and therefore it has $(n-k)$ real positive zeros in $(0,\infty)$ with respect to the weight function $\displaystyle{{(bq;q)_z c^z q^{\binom{z}{2}} \over (-bcq;q)_z (q;q)_z}{\left(-bcq^{z-k+1};q\right)_k\over q^{kz}}}.$
\end{proof}

\subsection{Quasi-orthogonality using Al-Salam Carlitz I polynomials}

\begin{theorem}
        Let $n\in\nn,\ k=1,\cdots,n-1,\ a<0,$ $\gamma\in\rr$ and $\gamma,\ \gamma+k\notin\{0,-1,-2,\cdots,-n\}.$
Then the polynomial $\varPhi_n^{(k)}(z)$
is quasi-orthogonal of order $k$ with respect to the weight function $\displaystyle{\left(qz,{qz\over a};q\right)_\infty}$
on the interval of orthogonality $(a,1)$ and has at least $(n-k)$ distinct real positive zeros.
\end{theorem}

\begin{proof}
Letting $r=3,s=2$ and $\alpha_2=\gamma+k-1,\alpha_3=z^{-1},\beta_1=0,\beta_2=q^\gamma$ and $z=\displaystyle{qz\over a}$ in \eqref{s2t1e1}, we obtain
\begin{multline}
\varPhi_n^{(k)}(z) = A_0\ _2\phi_1\left(\begin{matrix}q^{-n},z^{-1}\\0\end{matrix};q,{qz\over a}\right)
+A_1\ _2\phi_1\left(\begin{matrix}q^{-n+1},z^{-1}\\0\end{matrix};q,{qz\over a}\right)+\cdots\\
+A_k\ _2\phi_1\left(\begin{matrix}q^{-n+k},z^{-1}\\0\end{matrix};q,{qz\over a}\right)
= \sum_{i=0}^k A_i (-a)^{-n+i} q^{-\binom{n-i}{2}}\ U_{n-i}(z;q).\label{as1}
\end{multline}
Upon multiplication of the Al-Salam Carlitz I weight $\displaystyle{\left(qz,{qz\over a};q\right)_\infty}$ and
the factor $z^t,$ and integrating both sides of \eqref{as1} with respect to $z$ over the support $(a,1),$ we obtain
\begin{equation*}
\int_a^1 z^t \displaystyle{\left(qz,{qz\over a};q\right)_\infty} \varPhi_n^{(k)}(z) dz
\begin{cases}=0; & t=0,\dots, n-k-1,\\ \neq 0; & t=n-k.\end{cases}
\end{equation*}
Hence the result follows.
\end{proof}

\section{Interlacing of zeros of quasi-orthogonal $q$-Laguerre polynomials}\label{Section-4}
The following theorem shows the quasi-orthogonality of the $q$-Laguerre polynomials.
\begin{theorem}\label{s4t1}
The $q$-Laguerre polynomials $L_n^{(\delta)}(z;q)$ are
quasi-orthogonal of order $j$ for $-1-j<\delta<-j;$
$j=1\cdots,n-1,$ on $(0,\infty)$ with respect to the weight
function $\displaystyle{\frac{z^{\delta+j}}{(-z;q)_\infty}}.$
\end{theorem}

\begin{proof}
Replacing $\delta$ by $\delta+1$ in \eqref{s1e2}, we obtain
\begin{equation}\label{s4t1e1}
L_{n}^{(\delta)}(z,q)=q^{-n}{L}_{n}^{(\delta+1)}(z,q)-q^{-n}{L}_{n-1}^{(\delta+1)}(z,q).
\end{equation}
Iterating \eqref{s4t1e1} $k$ times, we find that
${L}_{n}^{(\delta)}(z,q)$ is a linear combination of the
polynomials ${L}_{n}^{(\delta+k+1)}(z,q),$
${L}_{n-1}^{(\delta+k+1)}(z,q),\cdots,{L}_{n-k-1}^{(\delta+k+1)}(z,q).$
That is,
\begin{equation*}
{L}_{n}^{(\delta)}(z,q)=\sum_{i=-1}^k
a_{n-i-1}{L}_{n-i-1}^{(\delta+i+1)}(z,q);\quad k=0,1,\cdots,n-2.
\end{equation*}
For $j=1,$ multiplication of $z^{k_1};\ k_1=0,1,\cdots,n-2$ and $\displaystyle{\frac{z^{\delta+1}}{(-z;q)_\infty}}$ in \eqref{s4t1e1},
and integrating the resulting equation with respect to $z$ over the support $(0,\infty),$ yields
\begin{multline}
\int_0^\infty z^{k_1}{L}_{n}^{(\delta)}(z,q)\frac{z^{\delta+1}}{(-z;q)_\infty}dz = \int_0^\infty q^{-n} z^{k_1}{L}_{n}^{(\delta+1)}(z,q)\frac{z^{\delta+1}}{(-z;q)_\infty}dz \\
- \int_0^\infty q^{-n}
z^{k_1}{L}_{n-1}^{(\delta+1)}(z,q)\frac{z^{\delta+1}}{(-z;q)_\infty}dz.\label{s4t1e2}
\end{multline}
The integrals on the right hand side of \eqref{s4t1e2} vanishes since $\delta+1>-1.$
Therefore, for $\delta\in(-2,-1),$ $L_{n}^{(\delta+1)}(z;q)$ is quasi-orthogonal of order $1$
with respect to the weight $\displaystyle{z^{\delta+1}\over (-z;q)_\infty}$ on $(0,\infty).$
Applying induction hypothesis, we have the required result for $j=\{1,\cdots,n-1\}.$
\end{proof}
Let $\nu_{1,n}<\nu_{2,n}<\cdots<\nu_{n,n}$ be the zeros of
${L}_n^{(\delta)}(z,q)$ and $\kappa_{1,n}<\kappa_{2,n}<\cdots<\kappa_{n,n}$ be
the zeros of ${L}_n^{(\delta+1)}(z,q)$. Then the following results
hold true.

\begin{theorem}\label{s4t2}
For $-2<\delta<-1,$ the zeros of ${L}_n^{(\delta)}(z,q)$ interlace
with the zeros of ${L}_n^{(\delta+1)}(z,q)$ and
${L}_{n-1}^{(\delta+1)}(z,q)$.
\end{theorem}

\begin{proof}
Clearly, the coefficients of ${L}_n^{(\delta+1)}(z,q)$ and ${L}_{n-1}^{(\delta+1)}(z,q)$ in \eqref{s4t1e1} are continuous and
have constant sign on $(0,\infty).$
Hence, \cite[Lemma 4(a)]{PKRS} implies all the zeros of ${L}_n^{(\delta)}(z,q)$ are real, simple and interlace with ${L}_n^{(\delta+1)}(z,q)$
as $\nu_{1,n}<\kappa_{1,n}<\nu_{2,n}<\kappa_{2,n}<\cdots<\nu_{n,n}<\kappa_{n,n}.$
Again from \cite[(4.14)]{Moak-1981}, we have
\begin{equation*}
{L}_{n-1}^{(\delta+1)}(z,q)=
\frac{(1-q^{n})q^{-n-\delta}}{z(q-1)}{L}_{n}^{(\delta)}(z,q)+\frac{(1-q^{n+\delta})q^{-n-\delta}}{z(1-q)}{L}_{n-1}^{(\delta)}(z,q).
\end{equation*}
Further applications of \cite[Lemma 4(b)]{PKRS} shows that
the real, simple zeros of ${L}_{n-1}^{(\delta+1)}(z,q)$ interlace
with those ${L}_{n}^{(\delta)}(z,q)$ as
$\nu_{1,n}<\kappa_{1,n-1}<\nu_{2,n}<\kappa_{2,n-1}<\cdots<\nu_{n-1,n}<\kappa_{n-1,n-1}<\nu_{n,n}.$
\end{proof}


\begin{thebibliography}{999}
\bibitem{Brezinski-2010} C. Brezinski. From numerical quadrature to Pad\'{e} Approximation, {\it Appl. Numer. Math.}, 60, 1209--1220 (2010).
\bibitem{Brezinski} C. Brezinski. Pad\'{e}-Type approximant and general orthogonal polynomials, Birkhauser, Basel, (1980).
\bibitem{BrezinskiDR} C. Brezinski, K. A. Driver, M. Redivo-Zaglia. Quasi-orthogonality with applications to some families of classical orthogonal
    polynomials, {\it Appl. Numer. Math.}, 48(2), 157--168  (2004).
\bibitem{BrezinskiDR-2019} C. Brezinski, K. A. Driver, M. Redivo-Zaglia. Zeros of quadratic quasi-orthogonal order 2 polynomials, {\it Appl. Numer. Math.}, 135, 143--145 (2019).
\bibitem{bu} Bustamante J., Mart\'{i}nez-Cruz R., Quesada J.M., Quasi orthogonal Jacobi polynomials and best one-sided $L_1$ approximation to step functions, {\it J. Approx. Theory} 198, 10--23 (2015).
\bibitem{Chihara} T. S. Chihara. On quasi-orthogonal polynomials, {\it Proc. Amer. Math. Soc.}, 8, 765--767 (1957).
\bibitem{Dickinson} D. Dickinson. On quasi-orthogonal polynomials, {\it Proc. Amer. Math. Soc.}, 12, 185--194 (1961).
\bibitem{Draux1} A. Draux. Polyn\^{o}mes orthogonaux formels-applications, {\it In Lecture Notes in Mathematics}, Vol. 974, Springer-Verlag, Heidelberg, (1983).
\bibitem{Draux3} A. Draux. On quasi-orthogonal polynomials, {\it J. Aprox. Theory}, 62(1), 1--14 (1990).
\bibitem{Fejer} L. Fej\'er. Mechanische Quadraturen mit Positiven Cotesschen Zahlen, {\it Math. Z.}, 37, 287--309 (1933).
\bibitem{Driver-2017} K. Driver, A. Jooste, Interlacing of zeros of quasi-orthogonal Meixner polynomials, {\it Quaest. Math.}, 40, 477--490 (2017).
\bibitem{Driver-2018} K. Driver, K. Jordaan. Zeros of Jacobi polynomials $P_n^{(\alpha,\beta)},\ -2<\alpha,\beta<-1$, {\it Numer. Algor.}, 79(4), 1075--1085 (2018).
\bibitem{DJ} K. Driver, K. Jordaan. Zeros of quasi-orthogonal Jacobi polynomials, {\it SIGMA,} 12, Paper No. 042, 13 pp, (2016).
\bibitem{DM1} K. Driver, M. E. Muldoon. Zeros of quasi-orthogonal ultraspherical polynomials, {\it Indag. Math. (N.S.)}, 27, 930--944 (2016).

\bibitem{DM3} K. Driver, M. E. Muldoon. Interlacing properties and bounds for zeros of some quasi-orthogonal Laguerre polynomials, {\it Comput. Methods Funct. Theory} 15, 645--654 (2015).

\bibitem{Gasper-book} G. Gasper, M. Rahman. {\it Basic Hypergeometric Series}, Encyclopedia of Mathematics and its Applications, 35,
    Cambridge Univ. Press, Cambridge, (1990).
\bibitem{PKRS} P. Gochhayat, K.H. Jordaan, K. Raghavendar, A. Swaminathan. Interlacing properties and bounds for zeros of ${}_2\phi_1$ hypergeometric and little $q$-Jacobi polynomials, {\it Ramanujan J.}, 40, 45--62 (2016).
\bibitem{Zhang} B.-Y. Guo, T. Sun, C. Zhang. Jacobi and Laguerre quasi-orthogonal approximations and related interpolations, {\it Math. Comp.} 82(281), 413--441 (2013).
 \bibitem{Hahn} W. Hahn. \"Uber Orthogonalpolynome, die $q$-Differenzengleichungen gen\"ugen, {\it Math. Nachr.}, {2}, 4--34 (1949).
\bibitem{Heine} E. Heine. Untersuchungen \"{u}ber die Reihe
    $1+\frac{(1-q^{\alpha})(1-q^{\beta})}{(1-q)(1-q^{\gamma})}x+\frac{(1-q^{\alpha}) (1-q^{\alpha+1})(1-q^{\beta})(1-q^{\beta+1})}{(1-q)(1-q^2)(1-q^{\gamma})(1-q^{\gamma+1})}x^2+\dots$,
    {\it J. Reine Angew. Math.}, 34, 285--328 (1847).
    \bibitem{IW} M. E. H. Ismail, X. Wang, On quasi-orthogonal polynomials: Their differential equations, discriminants and electrostatics. {\it J. Math. Anal. Appl.} 474(2), 1178--1197 (2019).
\bibitem{Johnston-2015} S. J. Johnston, K. Jordaan. Quasi-orthogonality and zeros of some $_2F_2$ and $_3F_2$ polynomials, {\it Appl. Numer. Math.}, 90, 1--8, (2015).
\bibitem{Johnston-2016} S. J. Johnston, A. Jooste, K. Jordaan. Quasi-orthogonality of some hypergeometric polynomials, {\it Integral Transforms Spec. Funct.}, 27(2), 111--125 (2016).
\bibitem{JT} K. Jordaan, F. To\'{o}kos. Orthogonality and asymptotics of Pseudo-Jacobi polynomials for non-classical parameters, {\it J. Approx. Th.}, 178, 1--12 (2014).
\bibitem{Joulak} H. Joulak. A contribution to quasi-orthogonal polynomials and associated polynomials, {\it Appl. Numer. Math.}, 54(1), 65--78 (2005).
 \bibitem{Kim} Y. S. Kim, A. K. Rathie, J. Choi. Three-term contiguous functional relations for basic hypergeometric series $_2\phi_1$, {\it Commun. Korean Math. Soc}, 20(2), 395–403, (2005).
\bibitem{Koekoek} R. Koekoek, P. A. Lesky, \ R. F. Swarttouw. {\it Hypergeometric Orthogonal Polynomials and Their $q$-Analogues}, Springer Monographs in Mathematics, Springer Verlag, Berlin (2010).
\bibitem{CK1} C. Krattenthaler. Manual for mathematica package for handling basic hypergeometric series, (2002).
\bibitem{Moak-1981} D.S. Moak. The $q$-analogue of the Laguerre polynomials, {\it J. Math. Anal.}, 81, 20--47 (1981).
\bibitem{Riesz} M. Riesz. Sur le probleme des moments. III, {\it Ark. Mat. Astron. Fys.}, 17(16), 1--52 (1923).
\bibitem{Ronveaux} A. Ronveaux. Quasi-orthogonality and differential equations, {\it J. Comput. Appl. Math.}, 29, 143--248 (1990).
\bibitem{Shohat} J.A. Shohat. On mechanical quadratures, in particular, with positive coefficients,
{\it Trans. Amer. Math. Soc.}, 42(3), 461--496 (1937).
\bibitem {Alta-2018} D.D. Tcheutia, A.S. Jooste, W. Koepf. Mixed recurrence equations and interlacing properties for zeros of sequences of classical q-orthogonal polynomials, {\it Appl. Numer. Math.} 125, 86--102  (2018).
\bibitem {Alta-SIGMA18} D.D. Tcheutia, A.S. Jooste, W. Koepf. Quasi-Orthogonality of Some Hypergeometric and q-Hypergeometric Polynomials, {\it SIGMA} 14, 051, 26 pages (2018).
\bibitem{Y1} Y. Xu. Quasi-orthogonal polynomials, quadrature, and interpolation, {\it J. Math. Anal. Appl.} 182, 779--799 (1994).
\bibitem{Y2} Y. Xu. On zeros of multivariate quasi-orthogonal polynomials and Gaussian cubature formulae, {\it SIAM J. Math. Anal.}, 25, 991--1001 (1994).
\bibitem{Y3} Y. Xu. A characterisation of positive quadrature formulae, {\it Math. Comp.}, 62, 703--718 (1994).
\bibitem{Zarzo} A. Zarzo, R. J. Y\'a\~nez, A. Ronveaux and J. S. Dehesa. Algebraic and spectral
properties of some quasi-orthogonal polynomials encountered in quantum radiation, {\it
J. Math. Phys.} 36, 5179--5197 (1995).
\end{thebibliography}
\end{document}